\documentclass[12pt,a4paper,twoside]{article}

\usepackage{graphicx} 
\usepackage{etoolbox}

\usepackage{amssymb,amsmath,amsthm,amsfonts,latexsym} 
\usepackage{correctmathalign}
\usepackage[mathscr]{euscript}

\makeatletter
\DeclareRobustCommand*{\bfseries}{%
  \not@math@alphabet\bfseries\mathbf
  \fontseries\bfdefault\selectfont
  \boldmath
}
\makeatother

\usepackage[labelsep=quad,indention=10pt,format=hang]{caption}
\usepackage{subcaption}

\usepackage[pdftex,bookmarks,colorlinks=false]{hyperref}

\usepackage{geometry}
\usepackage[auth-lg,affil-sl]{authblk}
\usepackage{float} 
\usepackage{dpfloat}
\usepackage{array,multirow,multicol} 
\usepackage{hhline} %
\usepackage{booktabs} 
\usepackage{longtable} 
\usepackage{diagbox} 

\usepackage{cite}
\usepackage{url}
\usepackage{breakcites} 
\usepackage{breqn} 
\allowdisplaybreaks

\usepackage{cite}

\usepackage{color} 
\usepackage{colordvi}
\usepackage{colortab}
\usepackage{colortbl}
\usepackage{colorweb}

\usepackage{bullcntr}%
\usepackage{doi}
\usepackage{epic}
\usepackage{epigraph}

\usepackage{enumitem}

\usepackage{calc,tikz}
\usetikzlibrary{}
\usetikzlibrary{intersections, arrows, arrows.meta, automata, er, calc, mindmap, folding, positioning,patterns, decorations.markings, fit, 
shapes, matrix, positioning, shapes.geometric, through}
\tikzstyle{vertex}=[circle, draw, inner sep=1pt, minimum size=4pt]
\newcommand{\vertex}{\node[vertex]}

\tikzstyle{ann} = [fill=white,font=\footnotesize,inner sep=1pt]
\tikzstyle{arrow} = [thick,<-->,>=stealth]

\newcommand{\ideg}[1]{\mathrm{ideg}}
\newtheorem{theorem}{Theorem}[section]
\newtheorem{definition}[theorem]{Definition}
\newtheorem{lemma}[theorem]{Lemma}
\newtheorem{proposition}[theorem]{Proposition}
\newtheorem{corollary}[theorem]{Corollary}

\newtheoremstyle{casesty}         
{}                   
{}                   
{\upshape}           
{}                   
{\itshape}         
{:-}                   
{1em}                
{}                   
\theoremstyle{casesty}
\newtheorem{case}{Case}

\newtheoremstyle{subcasesty}         
{}                   
{}                   
{\upshape}           
{}                   
{\itshape}          
{:-}                   
{1em}                
{}                   
\theoremstyle{subcasesty}

\newtheoremstyle{schemesty}         
{}                   
{}                   
{\upshape}           
{}                   
{\itshape}          
{:-}                   
{1em}                
{}                   
\theoremstyle{schemesty}

\newcommand{\keywordsname}{Keywords}
\newenvironment{keywords}{
	\normalfont
	\list{}{
		\labelwidth0pt
		\leftmargin0pt \rightmargin\leftmargin
		\listparindent\parindent \itemindent0pt
		\parsep0pt
		}
	\item[\hskip\labelsep\bfseries\keywordsname:] \small}{\endlist}

\newcommand{\mscname}{MSC 2020}
\newenvironment{msc}{
	\normalfont
	\list{}{
		\labelwidth0pt
		\leftmargin0pt \rightmargin\leftmargin
		\listparindent\parindent \itemindent0pt
		\parsep0pt
		}
	\item[\hskip\labelsep\bfseries\mscname:] \small}{\endlist}

\definecolor{ududff}{rgb}{0.30196078431372547,0.30196078431372547,1}
\definecolor{xdxdff}{rgb}{0.49019607843137253,0.49019607843137253,1}

\usepackage{lscape}

\usepackage[all]{nowidow} 

\title{\sc $s$-Shunt Intersection Graph of a Graph}
\author{{\bf Vinny Susan Prebhath$^\ast$, Sudev Naduvath$^\dag$}}
\affil{\small Centre for Mathematical Needs\\Department of Mathematics\\ Christ University, Bangalore, India. \\ $^\ast${\tt vinny.prebhath@res.christuniversity.in}\\ 
$^\dag${\tt sudev.nk@christuniversity.in}
}

\date{}

\allowdisplaybreaks

\begin{document}
\maketitle

\hrule
\begin{abstract}
The intersection graph of a family of sets $\{S_{1},S_{2},\ldots,S_{n}\}$ is a graph whose vertex set is $\{S_{1},S_{2},\ldots,S_{n}\}$ and two distinct vertices are adjacent if the intersection of the corresponding sets is non-empty. Different types of intersection graphs have been studied depending on the nature of sets taken as the vertex set. A study on a particular type of intersection graph called $s$-shunt intersection graph, generated from the $s$-arcs of a given graph is initiated in this paper.
\vspace{0.2cm}

\end{abstract}\vspace{0.2cm}
\hrule

\section{Introduction}

For basic terminology of graph theory, we refer to \cite{harary2018graph} and for further concepts and results in intersection graph theory, see \cite{mckee1999topics}. A multiset is a family of sets $\{S_{1}, S_{2},\ldots, S_{n}\}$ which allows the possibility that $S_{i} = S_{j}$ even though $i \neq j$. The \emph{intersection graph} of a family of sets $\mathscr{F}=\{S_{1}, S_{2},\ldots, S_{n}\}$, denoted by $\Omega(\mathscr{F})$, is the graph whose vertex set is $\mathscr{F}$ and $S_{i}$ is adjacent to $S_{j}$ if and only if $i\neq j$ and $S_{i} \cap S_{j} \neq \emptyset$. A graph $G$ is an intersection graph if there exists a family $\mathscr{F}$ such that $G \cong \omega(\mathscr{F})$ where $\mathscr{F}$ is then called a \emph{set representation} of $G$. Every graph is an intersection graph of a family of subgraphs of a graph. Different types of intersection graphs have been introduced and studied depending on the nature of the sets $S_{1}, S_{2},\ldots,S_{n}$. For different types of intersection graphs, refer to \cite{mckee1999topics}. 

One of the most studied intersection graphs is the line graph of a graph. For a graph $G$, the \emph{line graph} $L(G)$ of a graph $G$ is an intersection graph defined on the edge set of $G$ and two distinct vertices of $L(G)$ are adjacent if and only if the two corresponding edges of $G$ share a common vertex. The concept of line graph, initially introduced in \cite{whitney1992congruent}, has been studied independently by many authors using different names such as interchange graph, derived graph, derivative, covering graph, edge-to-vertex dual and adjoint. Certain characterisations of line graphs was studied in \cite{ray1967characterization}. Forbidden characterisation of line graphs was studied in \cite{beineke1970characterizations}. 

The line graph of a graph $G$ has been generalised in different ways. One of the first studied generalisation of line graphs is the path graph which was introduced in \cite{broersma1989path}. A \emph{path graph} of a graph $G$, denoted by $P_{k}(G)$, is a graph with vertex set as the set of all paths on $k$ vertices of $G$ and two distinct vertices are adjacent if the corresponding paths form a $P_{k+1}$ or a $C_{k}$ in $G$. An \emph{$\ell$-link} of $G$ is a walk of length $\ell \geq 0$ in which consecutive edges are different. The \emph{$\ell$-link graph} of $G$ which is a generalisaton of line graphs and path graphs, is a graph whose vertex set is the set of all $\ell$-links of $G$ and two vertices are joined by $\mu \geq 0$ edges if the vertices correspond to two subsequences of each of $\mu$ $(\ell+1)$-links of $G$ (see \cite{jia2017hadwiger}). Some other similar graphs studied are $(P_{3};I)$ transformation of a graph $G$ (see \cite{carroll1995subgraph}), $k$-edge graph (see \cite{prisner1994common}) and $(X,Y)$-intersection graph (see \cite{cai1996generalization}).
 
A graph $G$ is said to be \emph{vertex-transitive} if the automorphism group of $G$ acts transitively on the vertex set of $G$. A \emph{permutation group} $X$ on $V$ is said to be transitive if for any two elements $x$ and $y$ of $V$, there exists an element $\sigma \in X$ such that $\pi$ maps $x$ to $y$. For an integer $s \geq 1$, an \emph{$s$-arc} in a graph $G$ is a sequence of $(s+1)$ vertices $(v_{1},v_{2},...,v_{s},v_{s+1})$ of $G$ such that for all $1 \leq i \leq s$, $v_{i} \sim v_{i+1}$ and $v_{i} \neq v_{i+2}$ (see \cite{app1}). In an $s$-arc, the same vertex can be repeated although in all cases, this will not happen. If the automorphism group $X$ acts transitively on the $s$-arcs of $G$, then $G$ is said to be $s$-arc transitive. A vertex-transitive graph is a $0$-arc transitive graph whereas an arc-transitive graph is a $1$-arc transitive graph (refer \cite{app1}). 

Let $s \geq 1$ and $\alpha = (v_{0},v_{1},v_{2},\ldots,v_{s})$ be an $s$-arc in a graph $G$. For ease of usage, we write $(v_{0},v_{1},v_{2},\ldots,v_{s})$ as $v_{0}v_{1}v_{2}\ldots v_{s}$. Then the head of the $s$-arc, denoted by $head(\alpha)$ is the $(s-1)$-arc $v_{1}v_{2}\ldots v_{s}$ and its tail, denoted by $tail(\alpha)$ is defined to be the $(s-1)$-arc $v_{0}v_{1}\ldots v_{s-1}$. If $\alpha$ and $\beta$ are $s$-arcs, then $\beta$ follows $\alpha$ if there is an $(s+1)$-arc, say $\gamma$, such that $head(\gamma) = \beta$ and $tail(\gamma)=\alpha$. This can be described as $\alpha$ being \emph{shunted onto} $\beta$, which means $\alpha$ can be pushed one step onto $\beta$ (see \cite{app1}). The arc graph, denoted by $X^{s}(G)$, is defined to be a digraph with vertex set as the set of all $s$-arcs of $G$ and a directed edge exists from $\alpha$ to $\beta$ if $\alpha$ can be shunted onto $\beta$ (see \cite{app1}).

Motivated by the above mentioned studies, we define a new type of intersection graph from the set of all $s$-arcs of $G$ and the properties of these derived graphs are studied. Even though vertices can be repeated in an $s$-arc, we consider only those $s$-arcs where every vertex in the sequence is distinct. This is because it becomes rather challenging when we consider the intersection graphs on the set of all $s$-arcs of $G$ that can be shunted onto another $s$-arc in $G$. Also, since vertices can be repeated in an $s$-arc, $s$ can take any arbitrary value and the intersection graph for large values of $s$ will always be a complete graph. Hence, our first attempt to study intersection graphs of $s$-arcs in $G$ are of those $s$-arcs on distinct vertices which can be shunted onto some other $s$-arcs in $G$ on an $(s+1)$-arc where every vertex is distinct and then try to extend these results to the intersection graph that will be defined on all $s$-arcs in $G$ that can be shunted onto other $s$-arcs in $G$. 

\section{s-Shunt Intersection Graphs}

\begin{definition} \label{Defn1}{\rm
	Let $S$ be the set of all $s$-arcs on distinct vertices in $G$. Then, the \textit{$s$-shunt intersection graph} of $G$ (in short, ssi-graph), denoted by $A_{s}(G)$, is the graph obtained as follows:
	\begin{enumerate}[label=(\roman*),left=0.5cm]
		\item An element $\alpha \in S$ belongs to $V(A_{s}(G))$ if there exists another $s$-arc $\beta \in S$ such that $\alpha$ can be shunted onto $\beta$ on an $(s+1)$-arc on distinct vertices.
		\item Two vertices $\alpha$ and $\beta$ in $V(A_{s}(G))$ are adjacent in $A_{s}(G)$ if and only if they intersect at a minimum of one vertex in $G$.
	\end{enumerate} 
}\end{definition}

Throughout the paper, unless mentioned otherwise, we keep $s \geq 1$. It is to be noted that the upper bound of $s$ as per Definition \ref{Defn1} depends on the graph classes under consideration. For any graph $G$ and any possible value of $s$, if $\alpha_{i} = u_{i}u_{i+1}\ldots u_{i+s}$ is an $s$-arc in $G$, then the $s$-arc obtained by reversing the order of the vertices of $\alpha_{i}$, denoted by $\alpha_{i}' = u_{i+s}u_{i+s-1}\ldots u_{i}$, is also an $s$-arc in $G$. By the left-end vertex and the right end-vertex of $\alpha_{i}$, we mean the vertices $u_{i}$ and $u_{i+s}$ of $\alpha_{i}$ respectively. Unless mentioned otherwise, the graphs are simple, undirected and finite. 

An example of the ssi-graph of a bistar $S_{3,2}$, where $s \in \{1,2\}$ is illustrated in Figure \ref{fig:Example}. The $2$-arc $x_{1}x_{2}v_{1}$ is not a vertex in $A_{2}(S_{3,2})$ since $x_{1}x_{2}v_{1}$ cannot be shunted onto any other $2$-arc in $S_{3,2}$. Note that the $3$-si graph does not exist for any bistar since for any $3$-arc in $S_{m,n}$, there is no $3$-arc in $S_{m,n}$ that can be shunted onto any other $3$-arc in $S_{m,n}$.

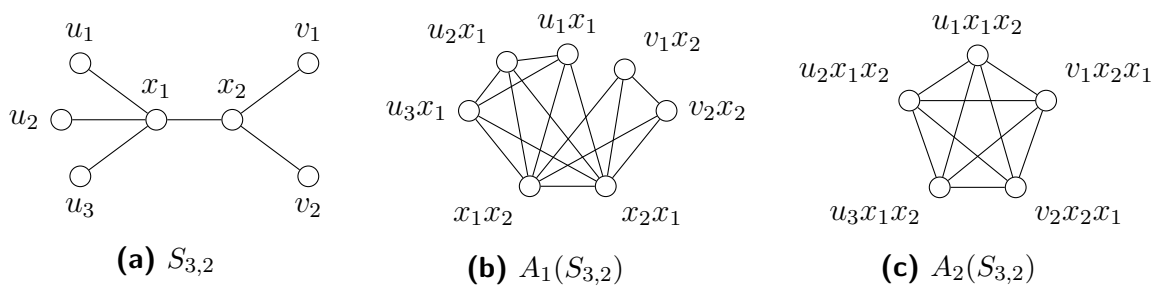
\begin{figure}[h]
	\centering
	\begin{subfigure}[h]{.30\textwidth}
		\centering
		\begin{tikzpicture}
			\vertex[inner sep=2.75pt](u1) at (0,0) [label=90:$u_{1}$]{};
			\vertex[inner sep=2.75pt](x1) at (1,-0.75) [label=90:$x_{1}$]{};
			\vertex[inner sep=2.75pt](x2) at (2,-0.75) [label=90:$x_{2}$]{};
			\vertex[inner sep=2.75pt](u2) at (-0.25,-0.75) [label=180:$u_{2}$]{};
			\vertex[inner sep=2.75pt](u3) at (0,-1.5) [label=270:$u_{3}$]{};
			\vertex[inner sep=2.75pt](v1) at (3,0) [label=90:$v_{1}$]{};
			\vertex[inner sep=2.75pt](v2) at (3,-1.5) [label=270:$v_{2}$]{};
			
			\draw (u1) edge node[] {} (x1);
			\draw (u2) edge node[] {} (x1);
			\draw (u3) edge node[] {} (x1);
			\draw (x1) edge node[] {} (x2);
			\draw (v1) edge node[] {} (x2);
			\draw (v2) edge node[] {} (x2);
			;
		\end{tikzpicture}
		\caption{$S_{3,2}$}	
	\end{subfigure}
	\hfill
	\begin{subfigure}[h]{.30\textwidth}
		\centering
		\begin{tikzpicture}
			\vertex[inner sep=2.75pt](u1) at (1.5,1) [label=90:$u_{1}x_{1}$]{};
			\vertex[inner sep=2.75pt](x1) at (1,-0.75) [label=245:$x_{1}x_{2}$]{};
			\vertex[inner sep=2.75pt](x2) at (2,-0.75) [label=295:$x_{2}x_{1}$]{};
			\vertex[inner sep=2.75pt](u2) at (0.7,0.9) [label=135:$u_{2}x_{1}$]{};
			\vertex[inner sep=2.75pt](u3) at (0.2,0.25) [label=180:$u_{3}x_{1}$]{};
			\vertex[inner sep=2.75pt](v1) at (2.25,0.8) [label=45:$v_{1}x_{2}$]{};
			\vertex[inner sep=2.75pt](v2) at (2.8,0.25) [label=0:$v_{2}x_{2}$]{};
			
			\draw (u1) edge node[] {} (x1);
			\draw (u2) edge node[] {} (x1);
			\draw (u3) edge node[] {} (x1);
			\draw (x1) edge node[] {} (x2);
			\draw (v1) edge node[] {} (x2);
			\draw (v2) edge node[] {} (x2);
			
			\draw (u1) edge node[] {} (x2);
			\draw (u2) edge node[] {} (x2);
			\draw (u3) edge node[] {} (x2);
			\draw (v1) edge node[] {} (x1);
			\draw (v2) edge node[] {} (x1);
			
			\draw (u1) edge node[] {} (u2);
			\draw (u1) edge node[] {} (u3);
			\draw (u2) edge node[] {} (u3);
			\draw (v1) edge node[] {} (v2);
			;
		\end{tikzpicture}
		\caption{$A_{1}(S_{3,2})$}
	\end{subfigure}
	\qquad
	\begin{subfigure}[h]{.30\textwidth}
		\centering
		\begin{tikzpicture}
			\vertex[inner sep=2.75pt](u1) at (1.5,0.25) [label=90:$u_{1}x_{1}x_{2}$]{};
			\vertex[inner sep=2.75pt](u2) at (0.6,-0.35) [label=135:$u_{2}x_{1}x_{2}$]{};
			\vertex[inner sep=2.75pt](u3) at (1,-1.5) [label=225:$u_{3}x_{1}x_{2}$]{};
			\vertex[inner sep=2.75pt](u4) at (2,-1.5) [label=315:$v_{2}x_{2}x_{1}$]{};
			\vertex[inner sep=2.75pt](u5) at (2.4,-0.35) [label=45:$v_{1}x_{2}x_{1}$]{};
			
			\draw (u1) edge node[] {} (u2);
			\draw (u2) edge node[] {} (u3);
			\draw (u3) edge node[] {} (u4);
			\draw (u4) edge node[] {} (u5);
			\draw (u5) edge node[] {} (u1);
			
			\draw (u1) edge node[] {} (u3);
			\draw (u2) edge node[] {} (u4);
			\draw (u3) edge node[] {} (u5);
			
			\draw (u1) edge node[] {} (u4);
			\draw (u2) edge node[] {} (u5);
			;
		\end{tikzpicture}
		\caption{$A_{2}(S_{3,2})$}
	\end{subfigure}
	\caption{A bistar and its $1$si-graph and $2$si-graph.}
	\label{fig:Example}
\end{figure}

The length of the longest path between two vertices $u$ and $v$ is the \emph{detour distance} of a graph $G$ which is denoted by $D(u,v)$ (refer \cite{harary2018graph}). The \emph{detour eccentricity} $e_{D}(v)$ of a vertex $v$ is the detour distance from $v$ to a vertex farthest from $v$. The maximum detour eccentricity among all vertices of $G$ is the \emph{detour diameter} of a graph which is denoted as $diam_{D}(G)$ (refer \cite{chartrand2004distance}). For ease of usage, we let $s^{\ast}$ be the detour diameter of a graph $G$. For example, in a cycle $C_{n}$, the possible values of $s$ are $1,2,\ldots,n-2$. In this case, $s^{\ast}=n-2$. Note that if 
$G$ is a tree, then $s^{\ast}=diam(G)$.

\begin{theorem} \label{Asconn1}
	The ssi-graph $A_{s}(G)$ of a graph $G$, if exists, is connected if and only if $G$ contains exactly one component with detour diameter at least $s+1$.
\end{theorem}
\begin{proof}
	Let $G$ be a connected graph with detour diameter at least $s+1$. Assume that the ssi-graph $A_{s}(G)$ is disconnected. For each $s$-arc $\alpha_{i}$ in $G$ that is a vertex of $A_{s}(G)$, there exists at least one other $s$-arc $\alpha_{j}$ in $G$ such that $\alpha_{i}$ can be shunted onto $\alpha_{j}$. Also, the $s$-arc $\alpha_{j}'$ can be shunted onto $\alpha_{i}'$ implying that each component of $A_{s}(G)$ consists of at least two vertices. Since $A_{s}(G)$ is disconnected, the vertices of the $s$-arc, say $\alpha$ of $G$, corresponding to the vertex $\alpha$ in a component $Q_{i}$ of $A_{s}(G)$, are not adjacent to any vertices of the $s$-arc, say $\beta$ of $G$, corresponding to the vertex $\beta$ in any other component $Q_{j};i\ne j$ of $A_{s}(G)$. However, there exists at least one path $P$ in $G$ with one end vertex in $\alpha$ and the other in $\beta$. Then, there exists a shuntable $s$-arc $\gamma$ in $G$ consisting of vertices of $\alpha,\beta$ and $P$. Hence, $\gamma$ is a vertex of $A_{s}(G)$ which is adjacent to both $\alpha$ and $\beta$, contradicting the fact that $\alpha$ and $\beta$ are in different components of $A_{s}(G)$. Therefore, the ssi-graph of a connected graph is always connected. Also, note that any component of $G$ having detour diameter less than $s+1$, if exists, does not contribute any vertex to $A_{s}(G)$. 
	
	
	If $A_{s}(G)$ is a non-empty connected graph, then there exists a path in $G$ between the end vertices of any two $s$-arcs of $G$, that can be shunted onto some other $s$-arc of $G$. For a graph $G$ with detour diameter at least $s+1$, every vertex of $G$ need not be a vertex of any $s$-arc of $G$, that can be shunted onto some other $s$-arc of $G$. If there exists a vertex $v$ of $G$ that is not a vertex of any of the $s$-arcs of $G$, then the detour eccentricity of $v$ is less than $s+1$. As a consequence, there may exist a component in $G$ with detour diameter less than $s+1$. 
	Therefore, all vertices of $A_{s}(G)$ must be $s$-arcs from the same component of $G$ with detour diameter at least $s+1$.
\end{proof}

Since the ssi-graph of any graph with detour diameter at most $s$ is a null graph, the results proved henceforth are not only true for connected graphs, but also for graphs $G$ where only one component of $G$ has detour diameter at least $s+1$. 

\begin{proposition} \label{Asconn2}
	For any connected graph $G$ of order $n$, $A_{s}(G)$ is complete when $s \geq \Big \lfloor \frac{n}{2} \Big \rfloor$.
\end{proposition}
\begin{proof}
	For $s \geq \Big \lfloor \frac{n}{2} \Big \rfloor$, intersection of any two $(s+1)$-sequences of vertices is non-empty. Hence, the resulting graph is a complete graph.
\end{proof}

The converse of this result need not be true. A counterexample for this can be seen in Figure \ref{fig:Example}.

One can observe that the graph $K_{1}$ is not realisable as an ssi-graph of any graph $G$. Suppose that there exists a graph $G$ for which $A_{s}(G) \cong K_{1}$. This implies that there exists only one $s$-arc, say $\alpha_{i}$, in $G$ which can be shunted onto some other $s$-arc $\alpha_{j}$ in $G$. But, by reversing the order of the vertices in $\alpha_{j}$, we get another $s$-arc in $G$ with the same vertex set as $\alpha_{j}$, that is a vertex of $A_{s}(G)$, implying that there exists two pairs of $s$-arcs in $G$, $(\alpha_{i},\alpha_{j})$ and $(\alpha_{j}',\alpha_{i}')$, where the first $s$-arc in each pair can be shunted onto the other $s$-arc in the respective pair and not vice-versa, which contradicts our assumption. 

\begin{theorem} \label{Asacyclic1}
	The only acyclic ssi-graph is $K_{2}$.
\end{theorem}
\begin{proof}
	Let $S$ be the set of all $s$-arcs in $G$ that are vertices of $A_{s}(G)$. We will prove that if the ssi-graph of a graph $G$ has more than two vertices, then $A_{s}(G)$ must contain an induced $C_{3}$. For every $s$-arc $\alpha_{i}$ in $S$, there exists another $s$-arc $\alpha_{j}$ in $G$, such that $\alpha_{i}$ can be shunted onto $\alpha_{j}$. The $s$-arc obtained by reversing the order of the vertices of $\alpha_{j}$, denoted by $\alpha_{j}'$, can be shunted onto $\alpha_{i}'$ and hence $\alpha_{i}\sim \alpha_{j}'$ in $A_{s}(G)$. Since there exists at least one more $s$-arc in $G$, say $\alpha_{k}$, which is a vertex of $A_{s}(G)$, by Theorem \ref{Asconn1}, there exists a path in $G$, between the vertices corresponding to the $(s+1)$-arc $\alpha_{i} \cup \alpha_{j}'$ and the $s$-arc $\alpha_{k}$ of $A_{s}(G)$, in $G$, leading to the following cases.
	
	There exists a third vertex $\alpha_{r} \sim \alpha_{i}$ (or $\alpha_{r} \sim \alpha_{j}'$) such that $\alpha_{r}$ can be shunted onto either $\alpha_{i}$, $\alpha_{i}'$, $\alpha_{j}$, $\alpha_{j}'$ or some other $s$-arc in $G$. We only consider the case when $\alpha_{r} \sim \alpha_{i}$ since the same argument applies for the case when $\alpha_{r} \sim \alpha_{j}'$. In the first case, the union of the $s$-arcs $\alpha_{r}$, $\alpha_{i}$ and $\alpha_{j}$ form an $(s+2)$-arc in $G$. This implies that $\alpha_{i}'$ is a vertex of $A_{s}(G)$ and hence $\alpha_{i} \sim \alpha_{r}$, $\alpha_{i} \sim \alpha_{i}'$ and $\alpha_{i}' \sim \alpha_{r}$, inducing a $C_{3}$ in $A_{s}(G)$. A similar argument applies to the case when $\alpha_{r}$ can be shunted onto $\alpha_{j}'$. For the case where $\alpha_{r}$ can be shunted onto $\alpha_{i}'$ (or $\alpha_{j}$), it follows that both $\alpha_{j}'$ (or $\alpha_{i}$) and $\alpha_{r}$ have the same head. Hence, the intersection between the vertices of the $s$-arcs $\alpha_{i}$, $\alpha_{j}'$ and $\alpha_{r}$ is always non-empty. In the latter case, there exists another $s$-arc $\alpha_{p}$ such that $\alpha_{r}$ can be shunted onto. If the left end-vertex of $\alpha_{r}$, say $v_{i}$, is the only common vertex of $\alpha_{i}$ and $\alpha_{r}$, then the degree of $v_{i}$ is at least $2$ where exactly one vertex of $\alpha_{i}$ belongs to $N(v_{i})$. As a result of this adjacency, $\alpha_{r}'$ can be shunted onto another $s$-arc in $G$ which proves the existence of an induced $C_{3}$ in $G$. If any other vertex of the $s$-arc $\alpha_{r}$ is the only common vertex of $\alpha_{i}$ and $\alpha_{r}$, then $\alpha_{i} \sim \alpha_{p}$ inducing a $C_{3}$ in $G$. It is also obvious that $\alpha_{i} \sim \alpha_{p}$ if the number of common vertices in $\alpha_{i}$ and $\alpha_{r}$ is greater than $1$. Hence, the ssi-graph of $G$ will always contain an induced $C_{3}$.

	Since for any ssi-graph on at least three vertices, a cycle on three vertices is an induced subgraph of $A_{s}(G)$, we can conclude that the only acyclic ssi-graph is isomorphic to $K_{2}$.
\end{proof}

\begin{lemma} \label{As*}
	The $(s^{\ast}-1)$-shunt intersection graph of $G$ is a complete graph.
\end{lemma}
\begin{proof}
	In view of the definition of $s^{\ast}$-arc of a graph, two distinct $s^{\ast}$-arcs of $G$ will share at least one common vertex in $G$ such that an $(s^{\ast}+1)$-arc is not contained in $G$. Other $s^{\ast}$-arcs, if any, can be obtained from the adjacency of an $s^{\ast}$-arc and an $s$-arc where $1 \leq s \leq s^{\ast}$. It can be noted that there exists no $l$-arc in the union of any of these $s^{\ast}$-arcs; with $l \geq s^{\ast}+1$. As a consequence of this, these $s^{\ast}$-arcs cannot be shunted onto any other $s^{\ast}$-arc and hence $A_{s^{\ast}}(G)$ is an empty graph. For each $s^{\ast}$-arc in $G$, there exists two $(s^{\ast}-1)$-arcs such that one can be shunted onto the other. 
	
	Between any two $(s^{\ast}-1)$-arcs in $G$, there exists at least one vertex that is common to the corresponding $(s^{\ast}-1)$-arcs in $G$. If $G$ contains two $(s^{\ast}-1)$-arcs that have no vertex in common, since $G$ is connected, there exists at least one edge that is adjacent to one vertex in both the $(s^{\ast}-1)$-arcs implying that an $l$-arc of length greater than $s^{\ast}$ is induced in $G$. Hence, the set of all these $(s^{\ast}-1)$-arcs will contain a vertex at which these $s^{\ast}$-arcs intersect in $G$. Therefore, the $(s^{\ast}-1)$-shunt intersection graph of $G$ is a complete graph.
\end{proof}

\begin{theorem} \label{Asacyclic2}
	The ssi-graph of a graph $G$ is acyclic if and only if there exists a unique $s^{\ast}$-arc in $G$. 
\end{theorem}
\begin{proof}	
	In view of Theorem \ref{Asacyclic1}, $K_{2}$ is the only acyclic ssi-graph. Suppose that the ssi-graph of a graph $G$ is $K_{2}$. This implies that there exists only two $s$-arcs in $G$ which are vertices of $A_{s}(G)$. If $\alpha_{i}$ is a vertex of $A_{s}(G)$ that can be shunted onto $\alpha_{j}$, then the second vertex of $A_{s}(G)$ is $\alpha_{j}'$. 
	We will show that $\alpha_{i} \cup \alpha_{j}$ must be the unique $s^{\ast}$-arc in $G$. 
	
	Suppose that there exists more than one $s^{\ast}$-arc in $G$. By Lemma \ref{As*}, the $(s^{\ast}-1)$-shunt intersection graph of $G$ is a complete graph. 
	Since each $s^{\ast}$-arc in $G$ contains two pairs of $(s^{\ast}-1)$-arcs such that only one can be shunted onto the other in each pair, $G$ must contain a unique $s^{\ast}$-arc which is $\alpha_{i} \cup \alpha_{j}$. If no additional $s^{\ast}$-arc is spanned in $G$, $G$ many have other $(s^{\ast}-1)$-arcs with one end vertex that is not a vertex of the unique $s^{\ast}$-arc of $G$. However, these $(s^{\ast}-1)$-arcs are not vertices of $A_{s^{\ast}-1}(G)$ because of the existence of a unique $s^{\ast}$-arc in $G$ and the shunting property. Since $K_{1}$ and $K_{2}$ are the only acyclic complete graphs, the result follows.
	
	Conversely, if there exists a unique $s^{\ast}$-arc in $G$, then there are two $(s^{\ast}-1)$-arcs which are subgraphs of the $s^{\ast}$-arc that are vertices of $A_{s^{\ast}-1}(G)$. If there exist $(s^{\ast}-1)$-arcs other than these in $G$, then for each of these $(s^{\ast}-1)$-arcs, there exists no other $(s^{\ast}-1)$-arc such that one can be shunted onto the other. Hence, any $(s^{\ast}-1)$-arc other than the two $(s^{\ast}-1)$-arcs in the unique $s^{\ast}$-arc are not vertices of the $(s^{\ast}-1)$si-graph of $G$. Therefore, $A_{s^{\ast}-1}(G) \cong K_{2}$. 
\end{proof}

\begin{proposition} \label{Asgirth}
	If $G$ is a graph which has more than one $s^{\ast}$-arc, then the corresponding ssi-graph is of girth $3$.
\end{proposition}
\begin{proof}
	As a consequence of Theorem \ref{Asacyclic2}, the ssi-graph of $G$ is never acyclic. As explained in the proof of Theorem \ref{Asacyclic1}, a $C_{3}$ is an induced subgraph of the ssi-graph of $G$. Hence, the girth of $A_{s}(G)$ is $3$.
\end{proof}

\subsection{$1$si and $2$si-Graphs of a Graph}

The order, size and degree of the $1$si-graph and the $2$si-graph of a graph $G$ are discussed in the following results. Further properties of the $1$si-graph of a graph $G$ are also studied.

\begin{proposition} \label{Asnumber1}
	The order of the $1$si-graph of a connected graph $G$ on at least three vertices is $2m-m_{1}$ where $m_{1}$ is the number of pendant edges.
\end{proposition}
\begin{proof}
	Since the size of $G$ is $m$, we partition the edge set of $G$ into disjoint subsets $E_{1}=\{u_{i}u_{j}:\deg_{G}(u_{i})=1$ \text{or} $\deg_{G}(u_{j})=1\}$ and $E_{2}=E(G)-E_{1}$. Let $|E_{1}|=m_{1}$. For every edge $u_{i}u_{j} \in E_{2}$, the $1$-arcs $u_{i}u_{j}$ and $u_{j}u_{i}$ are vertices of $A_{1}(G)$ because they can be shunted onto some other $1$-arc in $G$. Therefore, the number of these vertices is $2(m-m_{1})$. On the other hand, if $v_{i}v_{j}$ is an edge with $\deg_{G}(v_{i})=1$ (or $\deg_{G}(v_{j})=1$), then $v_{i}v_{j}$ (or $v_{j}v_{i}$) is a vertex in $A_{1}(G)$; but $v_{j}v_{i}$ (or $v_{i}v_{j}$) will not be a vertex in $A_{1}(G)$ since $v_{j}v_{i}$ (or $v_{i}v_{j}$) cannot be shunted onto any other $1$-arc in $G$. Hence, the number of these $1$-arcs in $G$ is $m_{1}$. Therefore, the order of $A_{1}(G)$ is $2m-m_{1}$.
\end{proof}

\begin{corollary} \label{Asnumber1cor}
	The order of the $1$si-graph of a graph $G$ with $\delta > 1$ is $2m$.
\end{corollary}
\begin{proof}
	The result is immediate from the fact that $E_{1}=\emptyset$ and $E_{2}=E(G)$.
\end{proof}

Recall that a graph $G$ on $n$ vertices determines $n$ cliques in the line graph of $G$. The $1$si-graph of $G$ exhibits this same property with the exception of the order of the cliques. This characterisation can be used to determine the number of edges in $A_{1}(G)$. 

\begin{theorem}
	The size of the $1$si-graph of a graph $G$ of order and size $n$ and $m$ respectively is $\frac{1}{2}\sum\limits_{i=1}^{n-m_{1}} (2d_{i}-m_{i})^2-3m+\frac{5m_{1}}{2}$, where $d_{i}$ is the degree of an internal vertex $u_{i}\in V(G)$, $m_{i}$ is the number of neighbours of $u_{i}$ that are pendant vertices of $G$ and $m_{1}$ is the total number of pendant vertices of $G$.
\end{theorem}
\begin{proof}
	Since the edges incident to any internal vertex of $G$ induce a clique of order $2d_{i}-m_{i}$ in $A_{1}(G)$, where $m_{i}$ is the number of neighbours of $u_{i}$ that are pendant vertices, the sum of the size of each of these cliques gives the size of $A_{1}(G)$. The $K_{2}$ induced by every edge in $G$ whose end-vertices are internal vertices is common to two cliques in $A_{1}(G)$. Hence, 
	\begin{equation*}
		\begin{split}
			|E(A_{1}(G))| &= \sum\limits_{i=1}^{n-m_{1}} \frac{(2d_{i}-m_{i})(2d_{i}-m_{i}-1)}{2}-(m-m_{1}) \\
			&=\sum\limits_{i=1}^{n-m_{1}} \frac{(2d_{i}-m_{i})^2}{2}-\sum\limits_{i=1}^{n-m_{1}} \Big (\frac{2d_{i}}{2}-\frac{m_{i}}{2} \Big) -m+m_{1}\\
			&=\sum\limits_{i=1}^{n-m_{1}} \frac{(2d_{i}-m_{i})^2}{2}-\sum\limits_{i=1}^{n-m_{1}} d_{i}+\sum\limits_{i=1}^{n-m_{1}}\frac{m_{i}}{2} -m+m_{1}\\
			&=\frac{1}{2}\sum\limits_{i=1}^{n-m_{1}} (2d_{i}-m_{i})^2-(2m-m_{1})+\frac{m_{1}}{2} -m+m_{1}\\
			&=\frac{1}{2}\sum\limits_{i=1}^{n-m_{1}} (2d_{i}-m_{i})^2-3m+\frac{5m_{1}}{2}.
		\end{split}
	\end{equation*} 
\end{proof}

\begin{theorem} \label{A1deg1}
	The degree of a vertex $uv$ in the $1$si-graph of a graph $G$ is 
	\begin{equation*}
		\deg_{A_{1}(G)}(uv)=
		\begin{cases}
			2(\deg_{G}(u)+\deg_{G}(v)-2)-\\(m_{u}+m_{v})+1, & \mbox{if} \ {\deg_{G}(u)>1} \ {and} \ {\deg_{G}(v)>1},\\
			2(\deg_{G}(v)-1)-m_{v}, & \mbox{if} \ {\deg_{G}(u)=1},\\
		\end{cases}
	\end{equation*}
	where $m_{u}$ and $m_{v}$ are the number of pendant neighbours of $u$ and $v$ respectively, other than $v$ and $u$, in $G$.
\end{theorem}
\begin{proof}
	Every $1$-arc $uv$ in $G$ is adjacent to the $1$-arcs with either $u$ or either $v$ as one of its end-vertices. If there exists a pendant vertex $w$ of $G$ in the neighbourhood of $u$, then $wu$ or $wu$ is a vertex in $A_{1}(G)$ that is adjacent to $uv$ in $A_{1}(G)$. However, for those internal vertices $x$ of $G$ in the neighbourhood of $u$ of $G$, since both $xu$ and $ux$ are vertices of $A_{1}(G)$, both of these vertices are neighbours of $uv$ in $A_{1}(G)$. Hence, if both $uv$ and $vu$ are vertices of $A_{1}(G)$, then the number of $1$-arcs in $G$ with $u$ as an end-vertex is $2(\deg_{G}(u)-1-m_{u})+m_{u}$. Therefore, $\deg_{A_{1}(G)}(uv)=2(\deg_{G}(u)-1)-m_{u} + 2(\deg_{G}(v)-1)-m_{v} + 1$. If one of the end-vertices of an edge $uv$ in $G$ is a pendant vertex, say $u$, then the degree of the $1$-arc $uv$ in $A_{1}(G)$ is $2(\deg_{G}(v)-1)-m_{v}$.
\end{proof}

\begin{theorem}
	If the $1$si-graph of a graph $G$ is regular, then $G$ is either a star or $\delta(G)>1$.
\end{theorem}
\begin{proof}
	It is enough to prove that if $G$ is a graph which contains both pendant edges and internal edges, then $A_{1}(G)$ is not regular. Consider a support vertex $u$ of $G$ with degree $k$. Let $vu$ be a pendant edge of $G$ and $vw$ be an internal edge of $G$. If there are $k_{1}$ number of pendant vertices with $u$ as the support vertex, then the degree of the vertex $vu$ in $A_{1}(G)$ is $(k_{1}-1)+2(k-k_{1})$, whereas the degree of $vw$ in $A_{1}(G)$ is at least $(k_{1}-1)+2(k-k_{1})+1$. Hence, $A_{1}(G)$ is not regular.
\end{proof}

\begin{theorem}
	The $1$si-graph of a graph $G$ with $\delta>1$ is regular if and only if $G$ is a regular graph or a biregular graph where the sum of the degrees of any two adjacent vertices of $G$ is equal.
\end{theorem}
\begin{proof}
	Since $A_{1}(G)$ is regular, it follows from Theorem \ref{A1deg1} that for any edge $uv$ in $G$, $2(\deg_{G}(u)+\deg_{G}(v)-2)+1=k$. Hence, $\deg_{G}(u)+\deg_{G}(v)=\frac{k+3}{2}$ for all edges in $G$. This implies that all vertices in the neighbourhood of $u$ are of the same degree. Also, if $w\in N(u)$, then for the $1$-arc $wx$ in $G$, $\deg_{G}(x)=\deg_{G}(u)$ and all other neighbours of $w$ should have the same degree. Hence, if $\deg_{G}(u)=\deg_{G}(v)$, then $G$ is regular; else, $G$ is biregular with $\deg_{G}(u)+\deg_{G}(v)=\frac{k+3}{2}$ for all edges in $G$. The converse part of the theorem is immediate from \ref{A1deg1}. 
\end{proof}


The order of the $P_{3}$-graph which was studied in \cite{broersma1989path} is mentioned in the following theorem. 

\begin{lemma} {\rm \cite{broersma1989path}}
	For a graph $G$, $|V(A_{2}(G))| = \sum\limits_{v \in I(G)} {deg(v) \choose 2}$, where $I(G)$ is the set of all vertices of $G$ with degree greater than $1$. 
\end{lemma}

Every path on three vertices is a $2$-arc in $G$. Hence, the above result can be used to find the order of the $2$si-graph of a graph $G$. We define the \emph{internal degree} of a vertex $u$ as the difference between the degree of the vertex in $G$ and the number of pendant vertices in $N[u]$, and denote it as $i$-degree or $\ideg{G}(u)$.

\begin{theorem} \label{Asnumber2}
	The order of the $2$si-graph of a connected graph $G$ of order $n$ is $$\sum\limits_{u_{i} \in I(G)} \Big (n_{i}\ideg{G}(u_{i})  + 2{\ideg{G}(u_{i}) \choose 2} - 2k_{i} - p_{i} \Big )$$, where $I(G)$ is the set of all internal vertices of $G$, $\ideg{G}(u_{i})$ is the internal degree of $u_{i}$, $n_{i}$ is the number of pendant vertices that are neighbours of $u_{i}$, $k_{i}$ and $p_{i}$ is the number of triangles with $u_{i}$ as the common vertex where both the neighbours of $u_{i}$ in these triangles are of degree $2$ in $G$ and only one neighbour of $u_{i}$ in these triangles is of degree $2$ in $G$ respectively.
\end{theorem}
\begin{proof}
	As mentioned earlier, only those $2$-arcs which can be shunted onto some other $2$-arc in $G$ for which the union of the two $2$-arcs is a $3$-arc on distinct vertices in $G$ are vertices in $A_{2}(G)$. Hence, if there exists a $2$-arc $u_{i}u_{j}u_{k}$ which can be shunted onto only one other $2$-arc $u_{j}u_{k}u_{i}$, then $u_{i}u_{j}u_{k}$ is not a vertex of $A_{2}(G)$, since $u_{i}u_{j}u_{k}u_{i}$ is not a $3$-arc where every vertex is distinct. Also, if the right end-vertex of a $2$-arc in $G$ is a pendant vertex in $G$, then this $2$-arc is not a vertex of $A_{2}(G)$ since it cannot be shunted onto any other $2$-arc in $G$. Therefore, depending on the degree of each neighbor of a vertex in $G$, we have the following types of vertices in $I(G)$. The first type of vertices are those vertices whose neighbours are pendant vertices of $G$. The second type of vertices are vertices of degree 2 in $G$ which are vertices of a $C_{3}$ and the remaining set of vertices are the third type of vertices.
	
	Consider the set of vertices in the set $I(G)$ which are adjacent to some pendant vertices of $G$. Since these vertices are the support vertices of $G$, few $2$-arcs with these support vertices as the middle vertex has a pendant vertex of $G$ as one end-vertex. If there are $n_{i}$ number of pendant vertices adjacent to the vertex $u_{i}$, then the number of $2$-arcs in $G$ with one end-vertex as the pendant vertex of $G$ and the support vertex $u_{i}$ as the middle vertex of the $2$-arc is $2n_{i}{\ideg{G}(u_{i}) \choose 1}$. The $1$-arc with the left end-vertex as the pendant vertex and the support vertex $u_{i}$ as the right end-vertex is the tail of a $2$-arc in $G$ which we shall denote as $\alpha$. This implies that the $\alpha$ is a vertex in $A_{2}(G)$ and the $2$-arc obtained by reversing the order of the vertices of $\alpha$ will not be a vertex in $A_{2}(G)$. Hence, the total number of vertices in $A_{2}(G)$ with the left end-vertex as the pendant vertex of $G$ and the support vertices as the middle vertex of the $2$-arc is $\sum\limits_{u_{i} \in I(G)} n_{i}{\ideg{G}(u_{i}) \choose 1}$.
	
	From the remaining set of vertices in $I(G)$, by choosing any two vertices from the neighbourhood of these vertices, we obtain the number of $2$-arcs with vertices in $I(G)$ as middle vertices and vertices of degree at least $2$ in $G$ as both the end-vertices of the $2$-arcs. Since the end-vertices of each of these $2$-arcs, say $u_{i}u_{j}u_{k}$ are not pendant vertices, both $u_{i}u_{j}u_{k}$ and $u_{k}u_{j}u_{i}$ can be shunted onto other $2$-arcs in $G$. Hence, the number of these $2$-arcs is $2\sum\limits_{u_{i} \in I(G)} {\ideg{G}(u_{i}) \choose 2}$.  
	
	Out of these remaining set of vertices $u_{i}$ in $I(G)$, we consider the set of vertices which are vertices of some $C_{3}$ in $G$. If both the neighbours of $u_{i}$ which are vertices of the same $C_{3}$ are not adjacent to any other vertex in $G$, then the $2$-arc with $u_{i}$ as the middle vertex, and these two neighbours as the end vertex will not be a vertex in $A_{2}(G)$. If $u_{i}$ is a vertex of $k_{i}$ number of such $C_{3}$'s, then we remove $2k_{i}$ number of such $2$-arcs from $2\sum\limits_{u_{i} \in I(G)} {\ideg{G}(u_{i}) \choose 2}$ number of $2$-arcs in $G$. Similarly, if only one neighbour of $u_{i}$ in a $C_{3}$ is not adjacent to any other vertex in $G$, then the $2$-arc with this neighbour as the right end-vertex and the third vertex of the same $C_{3}$ as the left end-vertex, which shall be denoted as $\beta$ cannot be shunted onto any other $2$-arc in $G$. However, $\beta'$ can be shunted onto some other $2$-arc in $G$ implying that this $2$-arc is a vertex in $A_{2}(G)$. If $u_{i}$ is a vertex of $p_{i}$ number of such $C_{3}$'s, then we remove $p_{i}$ number of $2$-arcs from $2\sum\limits_{u_{i} \in I(G)} {\ideg{G}(u_{i}) \choose 2}$ number of $2$-arcs in $G$. Hence, the total number of $2$-arcs which are vertices of the $2$si-graph of $G$ with vertices in $I(G)$ as middle vertices and vertices of degree at least $2$ in $G$ as both the end-vertices of the $2$-arcs are $2\sum\limits_{u_{i} \in I(G)} n_{i}{\ideg{G}(u_{i}) \choose 2} - 2k_{i} - p_{i}$. 
	
	Concluding from the above arguments, the total number of vertices in $A_{2}(G)$ is $\sum\limits_{u_{i} \in I(G)} \Big (n_{i}{\ideg{G}(u_{i}) \choose 1} + 2{\ideg{G}(u_{i}) \choose 2} - 2k_{i} - p_{i} \Big )$.
\end{proof}

\begin{corollary} \label{Asnumber22}
	The order of the $2$si-graph of a $C_{3}$-free graph $G$ is $\sum\limits_{u_{i} \in I(G)} \Big (n_{i}{\ideg{G}(u_{i}) \choose 1} + 2{\ideg{G}(u_{i}) \choose 2} \Big )$, where $I(G)$ is the set of all internal vertices of $G$, $\ideg{G}(u_{i})$ is the internal degree of $u_{i}$ and $n_{i}$ is the number of pendant vertices that are neighbours of $u_{i}$.
\end{corollary}
\begin{proof}
	The result follows as an immediate consequence of Theorem \ref{Asnumber2} when $k_{i}=p_{i}=0$.
\end{proof}

\begin{corollary} \label{Asnumber2cor1}
	The order of the $2$si-graph of a $C_{3}$-free graph $G$ with $\delta > 1$ is $2\sum\limits_{v \in V(G)} {deg(v) \choose 2}$.
\end{corollary}
\begin{proof}
	Since $n_{i}=k_{i}=p_{i}=0$, the total number of vertices in $A_{2}(G)$ is $2\sum\limits_{u_{i} \in V(G)} {deg(u_{i}) \choose 2}$.
\end{proof}

\begin{corollary} \label{Asnumber2cor2}
	The order of the $2$si-graph of any connected graph $G$ with $\delta > 2$ is $2\sum\limits_{v \in V(G)} {deg(v) \choose 2}$.
\end{corollary}
\begin{proof}
	For any graph $G$ with $\delta > 2$, even though $G$ may contain a $C_{3}$ as an induced subgraph, $k_{i}=p_{i}=0$. Also $n_{i}=0$ since $G$ does not contain any pendant vertices. Therefore, the total number of vertices in $A_{2}(G)$ is $2\sum\limits_{u_{i} \in V(G)} {deg(u_{i}) \choose 2}$.
\end{proof}

It is rather challenging to find the size of $A_{2}(G)$ and hence we avoid this result. We now look at a few characterisations of the $1$si-graph of a graph $G$. It is interesting to note that the line graph of $G$ is an induced subgraph of $A_{1}(G)$ which will be proved in the results that follow. Further, various properties of $L(G)$ are used to obtain some characterisations of $A_{1}(G)$.

\begin{lemma} \label{1arcsubgraph}
	For every graph $G$ on at least three vertices, $L(G)$ is an induced subgraph of the $1$si-graph of $G$.
\end{lemma}
\begin{proof}
	Each edge $xy$ is a $1$-arc of the graph $G$. Since $G$ is connected, $xy$ is a vertex if it can be shunted onto another $1$-arc in $G$. If $xy$ cannot be shunted onto any other $1$-arc in $G$, then $G \cong K_{2}$ and $A_{1}(K_{2})$ is an empty graph. Since $G$ is a graph on at least three vertices, every edge $xy$ of $G$ is a vertex in $A_{1}(G)$. These vertices induce the line graph of $G$ in $A_{1}(G)$.
\end{proof}

\begin{theorem}
	The $1$si-graph of a graph $G$ is isomorphic to $L(G)$ if and only if $G$ is $K_{1,n}$.
\end{theorem}
\begin{proof}
	Suppose that the $1$si of a graph $G$ is isomorphic to $L(G)$. Then, the number of vertices in $A_{1}(G)$ is $m$. This is possible only if the vertex set of $A_{1}(G)$ has every edge of $G$ as its vertex set with no edge repeated in the collection. For every $1$-arc $xy$ in $G$ that can be shunted onto another $1$-arc in $G$, $yx$ will not be a vertex in $A_{1}(G)$ only if $x$ is a pendant vertex in $G$. Since, there are $m$ number of $1$-arcs in $G$ and one end-vertex of each $1$-arc is a pendant vertex in $G$, in order for $G$ to be connected, $G$ will $K_{1,m}$.
	
	Conversely, if $G \cong K_{1,n}$ and $y$ is the central vertex of $K_{1,n}$, then for every $1$-arc $xy$ in $K_{1,n}$, there exists another $1$-arc $yz$ in $K_{1,n}$ such that $xy$ can be shunted onto $yz$. Also, the $1$-arc $zy$ can be shunted onto the $1$-arc $yx$ but the $1$-arcs $yx$ and $yz$ cannot be shunted onto any other $1$-arc in $G$ since $x$ and $z$ are pendant vertices of $G$. Hence, $xy$ and $zy$ are vertices in $A_{1}(G)$ implying that all edges are vertices in $A_{1}(G)$ without any repetition. By Definition \ref{Defn1}, it is obvious that $A_{1}(G)$ is $L(K_{1,m})$.
\end{proof}

\begin{corollary} \label{AsKn}
	A complete graph of order $n \geq 2$ vertices is realisable from the $1$si-graph of $K_{1,n}$. 
\end{corollary}
\begin{proof}
	Since $A_{1}(K_{1,n})$ is the line graph of $K_{1,n}$, the $1$si-graph of $K_{1,n}$ is a complete graph on $n$ vertices.
\end{proof}

Although the diameter of $K_{1,n}$ is $2$, the $2$-si graph of $K_{1,n}$ is an empty graph since there exists no $2$-arc in $K_{1,n}$ that can be shunted onto another $2$-arc in $K_{1,n}$.

Recall that the line graph of a $k$-regular graph $G$ is a $(2k-2)$-regular graph. The following result is an immediate consequence of this result.

\begin{lemma} \label{1arcregualr}
	If $G$ is a $k$-regular graph, $k>1$, then $A_{1}(G)$ is a $(4k-3)$-regular graph.
\end{lemma}
\begin{proof}
	Since $G$ has no pendant vertices, for every $1$-arc $xy$ in $G$, the $1$-arc $yx$ is also a vertex in $A_{1}(G)$. Hence, $A_{1}(G)$ will have two $L$ as an induced subgraph of $G$. Also, each vertex $xy$ in one copy of $L$ will be adjacent to all the vertices in the closed neighbourhood of $yx$ in the second copy of $L$. Since the line graph of a $k$-regular graph is $(2k-2)$-regular, the degree of each vertex in $A_{1}(G)$ is $(2k-2) + (2k-2) +1 = 4k-3$.
\end{proof}

\begin{corollary}
	The $1$si-graph of a complete graph $K_{n}$, $n \geq 3$, is a $(4n-7)$-regular graph with two copies of the $L(K_{n})$ as its induced subgraph.
\end{corollary}
\begin{proof}
	This is immediate from Lemma \ref{1arcsubgraph} and Theorem \ref{1arcregualr} since $k=n-1$. Since $k>1$, for every $1$-arc $xy$ in $E(K_{n})$, the $1$-arc $yx$ is also a vertex in $A_{1}(K_{n})$ inducing two copies of $L(K_{n})$ as induced subgraphs of $A_{1}(K_{n})$.
\end{proof}

The $1$si-graph and the $2$si-graph of a bistar is discussed below.  Recall that a bistar, denoted by $S_{m,n}$, is the graph obtained by joining the centre of two stars with an edge.

\begin{theorem}
	The graph $(K_{m} \cup K_{n})+P_{2}$ is realisable from the $1$si-graph of the bistar $S_{m,n}$.
\end{theorem}
\begin{proof}
	If $xy$ is the edge that is adjacent to all the pendant vertices of $S_{m,n}$, both $xy$ and $yx$ are vertices of $A_{1}(S_{m,n})$ since the degree of both $x$ and $y$ in $G$ is greater than $1$. The $1$-arcs with $x$ as the right end-vertex and the pendant vertices as the left end-vertex is a vertex in $A_{1}(S_{m,n})$ but the $1$-arc obtained by reversing the order of the vertices in these $1$-arcs is not a vertex in $A_{1}(S_{m,n})$ since there exists no vertex adjacent to the pendant vertex in $G$ for which this $1$-arc can be shunted onto. The same argument can be applied for those $1$-arcs with $y$ as the right end-vertex. All $m$ number of $1$-arcs with $x$ as the right end-vertex induce a completed graph on $m$ vertices and the $n$ number of $1$-arcs with $y$ as the right end-vertex induce a completed graph on $n$ vertices in $A_{1}(G)$. Since $xy$ and $yx$ are adjacent to all other $1$-arcs in $G$, the vertices $xy$ and $yx$ in $A_{1}(G)$ are adjacent to all vertices in both $K_{m}$ and $K_{n}$. Also, the components $K_{m}$ and $K_{n}$ are disjoint. Hence, the result follows. 
\end{proof}

\begin{proposition}
	The $2$si-graph of the bistar $S_{m,n}$ is a complete graph on $m+n$ vertices.
\end{proposition}
\begin{proof}
	The number of vertices in $A_{2}(G)$ is $m+n$ since for each pendant vertex, either the $1$-arc $xy$ or $yx$ will be the head of these $2$-arcs. Since $xy$ or $yx$ is the head of any $2$-arc in $G$, any two $2$-arcs will intersect at one vertex at least implying that $A_{2}(G)$ is a complete graph on $m+n$ vertices. 
\end{proof}

A subset $D \subseteq V(G)$ is a dominating set if every vertex in $G$ is either an element of $D$ or is adjacent to a vertex in $D$ (see \cite{haynes2023fundamentals}). A minimum dominating set is a minimal dominating set of minimum cardinality and the cardinality of the minimum dominating set in $G$ is called the domination number of $G$, denoted by $\gamma(G)$ (refer \cite{haynes2023fundamentals}). The domination number of the line graph of a graph $G$ is equal to the edge domination number of $G$ (see \cite{murugan2018domination}). Edge domination of a graph $G$, introduced in (see \cite{hedetniemi1977edge}), is a subset $S'$ of $E(G)$ for which every edge not in $S'$ is adjacent to some edge in $S'$. The edge domination number $\gamma'(G)$ (or $\gamma'$, in short) of $G$ is the minimum cardinality taken over all edge dominating sets of $G$. Since edges are vertices of $A_{1}(G)$, we can obtain the domination number of $A_{1}(G)$ from the edge domination number of $G$. The edge domination number of $G$ is the lower bound of the domination number of $A_{1}(G)$. 

\begin{theorem}
	For any connected graph $G$, $\gamma(A_{1}(G))=\gamma(L)=\gamma'(G)$.
\end{theorem}
\begin{proof}
	Since each vertex $u_{i}u_{j}$ in $A_{1}(G)$ dominates its neighbours as well as the closed neighbourhood of $u_{j}u_{i}$ and $L$ is an induced subgraph of $A_{1}(G)$, a dominating set of order $\gamma(L)$ is a dominating set of $A_{1}(G)$. It is obvious that at least $\gamma(L)$ vertices is required to dominate all vertices of $G$. Hence, $\gamma(A_{1}(G))=\gamma(L)=\gamma'(G)$. 
\end{proof}

\subsection{Some Isomorphism Problems of the ssi-Graph of a Graph}

When it comes to discussing about isomorphism of $A_{s}(G)$, there are many questions of interest that arise. Does there exist a simple connected graph $G$ for which $A_{s}(G) \cong G$? Are there pairs of non-isomorphic connected graphs $G_{1}$ and $G_{2}$ such that $A_{s}(G_{1}) \cong A_{s}(G_{2})$? In the case of the line graph of a graph $G$, the only graph for which $L(G) \cong G$ is a cycle $C_{n}$. Also, it is well known that $K_{1,3}$ and $K_{3}$ are the only pair of non-isomorphic connected graphs that have the same line graph (see \citeonline{whitney1992congruent}). 

It is obvious from Proposition \ref{Asconn2} and Corollary \ref{AsKn} that there are connected graphs $G$ where $A_{s}(G) \cong A_{1}(K_{1,n}); s \geq \lfloor \frac{n}{2} \rfloor$. We try to discuss the isomorphism between the ssi-graph of simple, connected graphs $G$ for same values of $s$ as well as different values of $s$.

\begin{theorem} \label{A1iso}
	The order of the $1$-shunt intersection of a graph $G$ on $n$ vertices is $n$ if and only of $G$ is a bistar $S_{m_{1},m_{2}}$. 
\end{theorem}
\begin{proof}
	For graphs $G$ where $m>n$, $G$ contains at least one cycle implying that the number of vertices in $A_{1}(G)$ includes a minimum of additionally three vertices apart from all the $m$ edges in $G$. Hence, the order of $A_{1}(G)$ is greater than $m$ which is greater than $n$. If $m=n$, then $A_{1}(G)$ contains $n$ vertices only if $2m-m_{1}=m$ implying that the number of pendant edges in $G$ is $m$. If every edge is a pendant edge, this would imply that $G$ is $K_{1,m}$, which is a contradiction. If $m<n$, then the order of $A_{1}(G)$ in this case is $n$ if $2m-m_{1}=m+1$ since $m=n-1$. This implies that $m_{1}=m-1=n-2$, which means that $G$ has only two vertices whose degree is strictly greater than $1$. Also, since $G$ is connected, these two vertices are either adjacent or there exists a path in $G$ joining these two vertices. Since the remaining vertices are of degree $1$, these two vertices are adjacent in $G$ and the remaining vertices are adjacent to exactly one of these vertices implying that $G$ is a bistar.
\end{proof}

\begin{corollary}
	There is no graph $G$ such that $G \cong A_{1}(G)$.
\end{corollary}
\begin{proof}
	This follows from Theorem \ref{A1iso}, since the size of $A_{1}(S_{m,n})$ is greater than $n-1$.
\end{proof}

For $s \geq 1$, it is rather challenging to check whether there exists a graph $G$ such that $G \cong A_{s}(G)$. We try to first characterise graphs for which the order of the $s$-shunt intersection graph is the order of the graph. This still does not solve the problem since there may be other parameters for which the graphs may not be isomorphic.

\begin{theorem} \label{A2Iso}
	A graph $G$ and its $2$-shunt intersection graph has the same order if $\delta \leq 2$. 
\end{theorem}
\begin{proof}
	Let $G$ be a graph of order and size $n$ and $m$ respectively such that the order of $A_{2}(G)$ is $n$ and assume to the contrary $\delta > 2$. Since $2\sum\limits_{v \in V(G)} {deg(v) \choose 2}=n$, it follows that $n$ is even and $\sum\limits_{v \in V(G)} (deg(v))^2 = n + 2m$ by the hand-shaking lemma. Also, since $\sum\limits_{v \in V(G)} (deg(v))^2 < (\sum\limits_{v \in V(G)} deg(v))^2$, it follows that $n < 2m(2m-1)$ implying that $G$ is always a disconnected graph, a contradiction. 
\end{proof}

\begin{proposition}
	There is no connected $C_{3}$-free graph $G$ with $\delta > 1$ such that the order of the $2$-arc intersection graph of $G$ is the order of $G$.
\end{proposition}
\begin{proof}
	Since the order of $A_{2}(G)$ is $2\sum\limits_{v \in V(G)} {deg(v) \choose 2}$, the result is clear from the arguments mentioned in the proof of Theorem \ref{A2Iso}.
\end{proof}

Due to various challenges faced in obtaining a condition for which the order of $A_{2}(G)$ is the order of $G$ or the size of $G$, where $G$ is either acyclic or girth of $G$ is $3$, we leave this as an open problem. There are certainly graphs for which the order of $A_{s}(G)$ is the order of $G$. A path on $2s+2$ vertices is such a graph. Similarly, a path on $2s+1$ vertices is a graph whose $s$-shunt intersection graph has order equal to the size of $G$. The first step in studying the isomorphism between $G$ and it's ssi-graph is to characterise graphs for which the order of $G$ equals the order of $A_{s}(G)$. We leave the problem of obtaining conditions for which the order of $A_{s}(G)$ for any possible values of $s$ equals the order or size of $G$ as an open problem. 

\section{Conclusion}

In this article, we have introduced a new intersection graph from the set of all $s$-arcs on distinct vertices of $G$ that can be shunted onto some other $s$-arc on distinct vertices of $G$ and the properties of these derived graphs are studied. Further, the $1$si-graph and the $2$si-graph of a graph $G$ has been investigated. Various problems related to the isomorphism of ssi-graphs with other graph classes are open. In addition to these, it is interesting to study the structural characteristics of ssi-graphs of a graph $G$ for any possible value of $s$. All these facts highlight the scope for further research in this area.

\bibliographystyle{abbrv}
\bibliography{myref}

\end{document}